\documentclass{article}
\usepackage{amsmath}
\usepackage{amsfonts}
\usepackage{bbold}
\usepackage{tikz-cd}
\usepackage[mathscr]{eucal}
\usepackage{amsthm}
\usepackage{amssymb}
\usepackage{hyperref}
\usepackage{cleveref}
\usepackage[margin=1.3in]{geometry}
\usepackage{comment}

\DeclareMathOperator{\homeo}{Homeo}
\DeclareMathOperator{\aut}{Aut}
\DeclareMathOperator{\Stab}{Stab}
\DeclareMathOperator{\Sym}{Sym}
\DeclareMathOperator{\ev}{ev}

\newtheorem{THM}{Theorem}

\newtheorem{thm}{Theorem}
\newtheorem{lem}[thm]{Lemma}
\newtheorem{cor}[thm]{Corollary}

\theoremstyle{definition}
\newtheorem{defn}[thm]{Definition}
\newtheorem{question}[thm]{Question}

\begin{document}
\title{The complex of cuts in a Stone space}
\author{Beth Branman and Robert Alonzo Lyman}
\maketitle
\begin{abstract}
  Stone's representation theorem asserts a duality
  between Boolean algebras on the one hand and \emph{Stone spaces,}
  which are compact, Hausdorff, and totally disconnected,
  on the other.
  This duality implies a natural isomorphism between the homeomorphism group
  of the space and the automorphism group of the algebra.
  We introduce a \emph{complex of cuts}
  on which these groups act,
  and prove that when the algebra is countable
  and the space has at least five points,
  that these groups are the full automorphism group of the complex.
  % We extend our results to \emph{Stone space systems,}
  % that is, finite nested collections of Stone spaces,
  % each successive one a closed subset of its parent.
\end{abstract}
A \emph{Boolean algebra} is a nonempty set $B$
equipped with a pair of binary operations
$\vee$ and $\wedge$, special elements $0$ and $1$,
and a unary operation $\lnot$ that satisfy simple axioms
inspired by the algebra of logical statements
(where $\vee$ and $\wedge$ are ``or'' and ``and'',
$0$ is ``false'' and $1$ is ``true'' and $\lnot$ is ``not'', respectively)
or the algebra of subsets of a set
(where $0$ is the empty set, $1$ is the entire set,
and the operations are union, intersection and complementation).
Boolean algebras turn out to be partially ordered---in fact, lattices;
in the cases where the elements of the algebra are subsets of a set,
the partial order is inclusion.

A \emph{Stone space} $E$ is a topological space which is compact,
Hausdorff and totally disconnected.
Stone spaces have a basis of clopen (that is, closed and open) sets.

From a Stone space, one can produce a Boolean algebra,
namely the algebra of clopen sets.
Since unions, intersections and complements of clopen sets are clopen,
this really is a Boolean algebra.

Conversely, from a Boolean algebra $B$, one can produce
(in the presence of the Axiom of Choice)
a Stone space $E(B)$,
namely the space of \emph{ultrafilters} on the algebra.
A basis of open sets is given by sets of the form
\[
  U_b = \{\omega : b \in \omega \}.
\]
It turns out that the complement of $U_b$ is $U_{\lnot b}$
(because ultrafilters have the \emph{maximality} property:
for each $b \in B$ and ultrafilter $\omega$, either $b \in \omega$
or $\lnot b \in \omega$ but not both).
Therefore these sets are clopen,
making the space $E(B)$ totally disconnected.
It is also an easy exercise to verify that this basis of open sets
generates a topology which is Hausdorff,
since distinct ultrafilters must disagree on some element of $B$.
Compactness requires the Axiom of Choice.

In~\cite{Stone},
Stone proved that given a Boolean algebra $B$,
the map $b \mapsto U_b$ is an isomorphism
from $B$ to the Boolean algebra of clopen subsets of $E(B)$.

From this fundamental fact,
several observations are possible.
\begin{enumerate}
\item The group of homeomorphisms of $E(B)$
  and the group of isomorphisms of $B$
  are abstractly isomorphic.
  In fact, there are natural \emph{topologies}
  to equip $\homeo(E)$ and $\aut(B)$ with,
  and these topologies really are the same topology.
\item Since closed subsets of compact sets are compact,
  every clopen subset of $E$ is covered by finitely many
  clopen basis elements,
  so $E(B)$ is second countable if and only if $B$ is countable.
\item In the case that either $E$ is second countable
  or $B$ is countable,
  we have that $E$ is a closed subset of the Cantor set
  and (by Stone duality) that $B$ is a quotient of
  the countable free Boolean algebra
  (which we might suppose we knew already by countability).
\item Again assuming countability of $B$
  or second countability of $E$,
  we have that $\homeo(E) = \aut(B)$
  are \emph{non-Archimedean Polish groups,}
  being the automorphism group of a countable structure,
  namely $B$.
  As such, they are closed subgroups of
  $\Sym(\mathbb{N})$, the group of bijections of a countable set.
\end{enumerate}

In this paper,
inspired by the \emph{curve complex}
on a surface (particularly the case of a surface of infinite topological type)
we introduce a complex
we call the \emph{complex of cuts} $\mathscr{C}$
on which $\homeo(E) = \aut(B)$ acts.

\begin{THM}
  Suppose $E$ is second countable and contains at least five points.
  The group of automorphisms of $\mathscr{C}$,
  equipped with the permutation topology,
  is topologically isomorphic to $\homeo(E) = \aut(B)$.

  When $E$ is finite, the diameter of the finite graph $\mathscr{C}$
  is at most four, while when $E$ is infinite,
  $\mathscr{C}$ is a countably infinite graph which has diameter two.
  In particular, it is connected.
\end{THM}

Briefly
a \emph{cut} in $E$ is a partition of $E$
into two disjoint clopen sets $U \sqcup V$.
A cut is \emph{peripheral} if either $U$ or $V$ has $0$ or $1$ element.
Two cuts $U\sqcup V$ and $U' \sqcup V'$
\emph{cross} if all four of the pairwise intersections
$U \cap U'$, $U \cap V'$, $V \cap U'$ and $V \cap V'$ are nonempty,
otherwise they are \emph{compatible.}
The \emph{complex of cuts} has as vertices the non-peripheral cuts in $E$,
with two cuts being adjacent in $\mathscr{C}$ when they are compatible.
To think of $\mathscr{C}$ as a complex,
one may, as is standard with simplicial graphs,
add a simplex when its $1$-skeleton is present.

For a second countable Stone space $E$,
we may embed $E$ in the sphere $S^2$
and draw a Jordan curve which separates points in $U$ from points in $V$
and similarly for $U'$ and $V'$.
Compatibility, one checks,
is equivalent to being able to draw these curves disjointly.

Although we will attempt to prefer the Stone space perspective,
it is instructive to consider the Boolean algebra perspective as well:
there a \emph{cut} in $B$ is a pair $\{ a, \lnot a \}$,
a cut is \emph{peripheral} if either of $a$ or $\lnot a$
is $0$ or an \emph{atom,} that is, an immediate successor of $0$
in the partial order on $B$.
Two cuts $\{a, \lnot a\}$ and $\{a', \lnot a' \}$
\emph{cross} if none of the meets
$a \wedge a'$, $a \wedge \lnot a'$, $\lnot a \wedge a'$,
$\lnot a \wedge \lnot a'$ are $0$,
otherwise they are \emph{compatible.}

Since homeomorphisms or automorphisms
preserve complements,
we see that $\homeo(E) = \aut(B)$
acts on $\mathscr{C}$
by the rule that $\Phi(U \sqcup V) = \Phi(U) \sqcup \Phi(V)$,
or that $\Phi(\{a, \lnot a\}) = \{ \Phi(a), \lnot\Phi(a) \}$.

\section{Stone duality}
\subsection{Preliminary definitions}
\begin{defn}\rm
A \emph{Boolean algebra}
is a nonempty set $B$
with distinguished elements $0$ and $1$,
and three operations $\vee$, $\wedge$ and $\lnot$,
satisfying the following axioms.
\begin{enumerate}
\item The binary operations $\vee$ and $\wedge$
  are associative and commutative.
\item For all $a$ and $b$ in $B$,
  we have $a \vee (b \wedge a) = a$
  and $a \wedge (b \vee a) = a$.
\item $1$ is an identity for $\wedge$,
  while $0$ is an identity for $\vee$.
\item $\vee$ distributes over $\wedge$ and vice versa.
\item $a \vee \lnot a = 1$ while $a \wedge \lnot a = 0$.
\end{enumerate}
\end{defn}

The algebra $B$ is at the same time a \emph{lattice,}
a partially ordered set with ``meets''
$a \wedge b$ and ``joins'' $a \vee b$,
where we say that $a \le b$ when $a \vee b = b$.
\begin{defn}\rm
A \emph{homomorphism} of Boolean algebras $f\colon A \to B$
is a map $f \colon A \to B$ of sets
which preserves meets and joins,
and sends $1 \in A$ to $1 \in B$
and $0 \in A$ to $0 \in B$.
It follows from the axioms that $f(\lnot a) = \lnot f(a)$ as well.
Thus there is a \emph{category} of Boolean algebras and homomorphisms.
The \emph{terminal} Boolean algebra is $\mathbb{1} = \{ 0 = 1\}$,
while the \emph{initial} one is $\mathbb{2} = \{ 0, 1 \}$.
\end{defn}

\begin{defn}\rm
A \emph{filter} $\omega$ on a Boolean algebra $B$
is a proper, nonempty subset
which is \emph{meet closed}
in the sense that if $a$ and $b$ are in $\omega$,
then so is $a \wedge b$,
and also \emph{upward closed}
in the sense that if $a \in \omega$ and $a \le b$
(that is, $a \vee b = b$)
then $b \in \omega$.
\end{defn}

\begin{defn}
A filter (which we require to be a \emph{proper} subset, recall)
is called an \emph{ultrafilter}
when it is maximal with respect to inclusion.
For a Boolean algebra,
this is equivalent to the condition that for each $b \in B$,
either $b \in \omega$ or $\lnot b \in \omega$,
but not both.
\end{defn}

\begin{lem}\label{lem:ultrafilter}
  Assuming Zorn's Lemma,
  every filter $\omega$ on a Boolean algebra $B$
  may be extended to an ultrafilter.
  A filter is an ultrafilter if and only if,
  for each $b \in B$,
  either $b \in \omega$ or $\lnot b \in \omega$,
  but not both.
  If $f \colon B \to \mathbb{2}$ is a homomorphism,
  the collection of $b \in B$ for which $f(b) = 1 \in \mathbb{2}$
  is an ultrafilter.
  Finally,
  suppose that $S \subset B$
  is a subset with the property that \emph{every} ultrafilter on $B$
  contains an element of $S$.
  Then $S$ has a finite subset $S_0$ with the same property.
\end{lem}

\begin{proof}
  Suppose first that $\omega$ is a filter.
  The collection of filters which contain $\omega$ as a subset
  is partially ordered by inclusion.
  Supposing $\omega_1 \le \omega_2 \le \cdots$ is a chain of filters containing $\omega$,
  set $\omega_\infty = \bigcup_{n = 1}^\infty \omega_n$.
  This is a proper nonempty subset of $B$,
  (since each $\omega_n$ is a proper subset, none of the $\omega_n$ contain $0$, for instance).
  If $a \in \omega_\infty$ and $b \in \omega_\infty$,
  then they are each contained in some $\omega_n$,
  which contains $a \wedge b$ on account of being a filter.
  Similarly we see that $\omega_\infty$ is upward closed,
  so $\omega_\infty$ is a filter which is an upper bound for our given chain.
  By Zorn's lemma, our poset then has maximal elements, which are ultrafilters by definition.

  Supposing $\omega$ is a filter with the property that for each $b \in B$,
  eithe $b \in \omega$ or $\lnot b \in \omega$ but not both,
  we see that if a set $S$ properly contains $\omega$,
  it must contain both $b$  and $\lnot b$ for some $b \in B$.
  Such a set cannot be a filter, since the meet closed property forces it to contain $0$
  and the upward closed property forces it to then fail to be a proper subset of $B$;
  therefore $\omega$ is an ultrafilter.

  Given a homomorphism $f \colon B \to \mathbb{2}$,
  the algebra of $\mathbb{2}$ implies that the preimage of $1$ under $f$
  is a filter such that for each $b \in B$, either $f(b) = 1$ or $f(\lnot b) = 1$
  but not both.

  Note that for any element $b \in B$,
  note that the set $F(b) = \{ a \in B : b \le a\}$ is a filter on $B$ provided $b \ne 0$,
  and that if $a \le B$, then $F(a) \supset F(b)$.
  Supposing that $S \subset B$
  has the property that every ultrafilter on $B$ contains an element from $S$,
  given a finite subset $S_0 = \{ b_1, \ldots, b_n \} \subset S$,
  consider the set $F(S_0) = F(\lnot b_1 \wedge \cdots \wedge \lnot b_n)$.
  If this set is a filter, it will fail to contain any $b_n$;
  indeed $S_0$ will fail to be our sought-for finite subset of $S$ precisely when $F(S_0)$ is a filter.
  Supposing towards a contradiction that every $F(S_0)$ is a filter,
  letting $S_0$ vary among the finite subsets of $S$, we see that
  this collection of filters is partially ordered:
  when $S_0 \subset S'_0$, we have $F(S_0) \subset F(S'_0)$.
  Since we saw that unions of chains of filters are filters,
  appealing to Zorn's lemma again,
  we see that actually there is an ultrafilter $\omega$ containing no element of $S$,
  providing the contradiction we seek.
\end{proof}

\subsection{Stone Duality Theorem}

The collection of ultrafilters on $B$
is a Stone space $E(B)$,
(that is, it is compact, Hausdorff and totally disconnected)
where the basic clopen sets are of the form $U_b = \{ \omega : b \in \omega\}$.

Notice that since an ultrafilter defines a map $\omega\colon B \to \mathbb{2}$,
we see that if $g\colon A \to B$ is a homomorphism of Boolean algebras
and $\omega \in E(B)$ is a point of the associated Stone space,
we naturally obtain a point $g^*\omega$ in $E(A)$
by the rule that $g^*\omega$
is the ultrafilter corresponding to the composition
\[
  \begin{tikzcd}
    A \ar[r, "g"] & B \ar[r, "\omega"] & \mathbb{2}.
  \end{tikzcd}
\]
Indeed, the assignment $\omega \mapsto g^*\omega$
yields a continuous map $g^*\colon E(B) \to E(A)$:
we compute that the preimage ${(g^*)}^{-1}(U_a)$ satisfies
\[
  {(g^*)}^{-1}(U_a) = \{ \omega \in E(B) : a \in g^*\omega \}
  = \{\omega \in E(B) : g(a) \in \omega \} = U_{g(a)}.
\]
This defines a \emph{contravariant functor}
$E$ from the category of Boolean algebras and homomorphisms
to the category of Stone spaces and continuous maps
sending $B$ to $E(B)$ and $g\colon A \to B$
to $g^*\colon E(B) \to E(A)$.

Conversely, if $S$ is a Stone space,
it has a basis of clopen sets,
and the clopen sets of $S$ organize into a Boolean algebra
$\Omega(S)$.
If $f\colon S \to T$ is a continuous map,
observe that if $U \subset T$ is open (respectively closed),
then $f^{-1}(U)$ is open (respectively closed)
in $S$ by continuity,
so we have a map $f^{-1}\colon \Omega(T) \to \Omega(S)$,
and this assignment defines a contravariant functor $\Omega$
from Stone spaces to Boolean algebras.

\begin{thm}[Stone duality]
  The functors $E$ and $\Omega$ are (contravariant)
  equivalences of categories,
  in the sense that there exist isomorphisms
  $\eta\colon B \to \Omega(E(B))$
  and $\epsilon\colon E(\Omega(S)) \to S$
  natural in their arguments.

  Consequently the (abstract) group of automorphisms
  of a Boolean algebra $B$ is isomorphic
  to the group of homeomorphisms of $E(B)$,
  and moreover,
  the \emph{permutation topology}
  on $\aut(B)$ and the \emph{compact--open topology}
  on $\homeo(E(B))$ agree.
\end{thm}

Before turning to the proof,
let us recall quickly the topologies in question
in the ``moreover'' statement above.
Since in order to understand the topology on a group
one needs only to know what a neighborhood basis of the identity is,
we give only neighborhoods of the identity.

For the permutation topology (on a set, say),
the sets $\Stab(b) = \{\varphi \in \aut(B) : \varphi(b) = b \}$
are declared open and form a subbasis for the topology.
Put another way,
a basic open set is a finite intersection of stabilizers,
or a set of the form $\{\varphi : \varphi(b_i) = b_i\}$
for some finite set $\{b_1,\ldots,b_n\}$ in $B$.

For the compact--open topology
on a (compact, say) topological space $S$,
the sets $V(K,U) = \{ \varphi \in \homeo(S) : \varphi(K) \subset U\}$
are declared open,
where $K$ is compact,
$U$ is open,
and this set is an identity neighborhood
when $K \subset U$.
Since we assume that $S = E(B)$ is compact,
this always gives $\homeo(S)$ the structure of a topological group.
(In general with this topology it is possible
for inversion to fail to be continuous~\cite{Dijkstra}).
When $S$ itself has a subbasis of open sets,
we may restrict ourselves to choosing $U$ from this subbasis
to produce a subbasis for the compact--open topology.

With respect to the compact--open topology on a compact space $S$,
the map $\ev \colon \homeo(S) \times S \to S$
defined as $(\varphi, x) \mapsto \varphi(x)$ is continuous.
In fact it is the strongest 
(or \emph{coarsest, pace} Arens, meaning it has the fewest open sets) 
such topology~\cite{Arens, ArensAnnals}.

Now we turn to the proof.

\begin{proof}
  The map $\eta\colon B \to \Omega(E(B))$
  sends an element $b$ to the basic open set
  $U_b$;
  the distinguished elements $0$ to $\varnothing$
  and $1$ to $E(B) \subset E(B)$.

  We claim that $U_a \cap U_b = U_{a \wedge b}$,
  and that $U_a \cup U_b = U_{a\vee b}$.
  Indeed, since ultrafilters are upward closed,
  the inclusions $U_{a\wedge b} \subset U_a$ and $U_b$
  are clear,
  as are the containments $U_{a\vee b} \supset U_a$ and $U_b$.
  Conversely, because ultrafilters are meet closed,
  we see that $U_a \cap U_b \subset U_{a \wedge b}$
  and similarly $U_a \cup U_b \supset U_{a \vee b}$.
  Therefore $\eta$ is a homomorphism.

  For each $b \in B$ not equal to $0$,
  the set $F(b) = \{ a \in B : a \ge b \}$ is a filter which contains $b$.
  By \Cref{lem:ultrafilter},
  we have that $F(b)$ is contained in an ultrafilter.
  By considering the element $a \wedge \lnot b$,
  from this fact it is clear that if $a \ne b$,
  then $U_a \ne U_b$,
  so the map $\eta$ is injective.

  Conversely, suppose that $U \subset E(B)$ is clopen.
  Since the $U_a$ (by definition)
  form a basis for the topology on $E(B)$, since $U$ is open,
  it contains some $U_{a_0}$.
  Since $U_{a_0}$ is closed,
  $U - U_{a_0}$ is open and therefore contains $U_{a_1}$ for some $a_1$.
  Proceeding inductively we have a cover of $U$ by disjoint open sets.
  By compactness of $E(B)$
  and hence of the closed subset $U$,
  finitely many of them cover,
  hence $U$ is equal by the previous observation
  to $U_{a_0 \vee \cdots \vee a_n}$.
  In other words, the map $\eta$ is surjective.

  On the other hand, suppose that $S$ is a Stone space.
  An element of $E(\Omega(S))$ is an ultrafilter on $\Omega(S)$,
  that is, it is a filter,
  i.e.\ a collection of nonempty clopen sets
  which is closed under (finite) intersections and directed unions,
  which is maximal with respect to inclusion.
  Since the clopen subsets of our ultrafilter
  are nonempty,
  we have that this collection $\omega$
  has the \emph{finite intersection property,}
  and hence by compactness of $S$,
  there exists at least one point in the intersection of all $U \in \omega$.

  We claim that in fact there is \emph{exactly} one point in this intersection.
  Indeed, suppose $x$ is contained in $\bigcap_{U \in \omega} U$.
  Since $S$ is Hausdorff,
  for any $y \ne x$,
  there is an open (in fact clopen) neighborhood $V$ of $x$ not containing $y$.
  By maximality of $\omega$, we must have that $V$ is contained in $\omega$,
  and so $y$ is not in $\bigcap_{U \in \omega} U$.

  If $x$ is the unique point of $\bigcap_{U \in \omega} U$,
  let $\epsilon\colon E(\Omega(S)) \to S$
  be the map $\omega \mapsto x$.
  Notice that if $\omega \ne \omega'$,
  then there is a clopen set $U \subset S$
  such that $U \in \omega$ but $S - U$ is in $\omega'$.
  It follows that $\epsilon(\omega) \ne \epsilon(\omega')$.
  It follows that the map $\epsilon$ is injective.

  If $x \in S$ is a point,
  observe that the collection of all clopen subsets of $S$
  containing $x$ is an ultrafilter on $\Omega(S)$:
  it is clearly a filter,
  and it is maximal because $S$ is Hausdorff and totally disconnected.
  It follows that the map $\epsilon$ is surjective.

  Finally, supposing that $U \subset S$ is clopen,
  consider $\epsilon^{-1}(U)$.
  This is the collection of all ultrafilters
  whose unique points of total intersection lie in $U$.
  But this is exactly the set of ultrafilters
  containing the clopen set $U$.
  It follows that $\epsilon$ is continuous.
  To see that is open (and thus a homeomorphism),
  suppose that $U \subset E(\Omega(S))$ is a basic open set,
  say $U = U_V = \{ \omega \in E(\Omega(S)) : V \in \omega\}$.
  It is clear that $\epsilon(U) \subset V$.
  In fact, each point of $V$ is,
  by our proof of surjectivity of $\epsilon$,
  hit by some element of $E(\Omega(S))$
  which is clearly contained in $U$,
  so we conclude that $\epsilon$ is open.

  That the maps $\eta$ and $\epsilon$
  are natural in their argument is the statement
  that for each map (in the correct category)
  $A \to B$ or $S \to T$,
  the following diagrams commute
  \[
    \begin{tikzcd}
      A \ar[r, "\eta_A"] \ar[d] & \Omega(E(A)) \ar[d] \\
      B \ar[r, "\eta_B"] & \Omega(E(B))
    \end{tikzcd}
    \qquad
    \begin{tikzcd}
      E(\Omega(S)) \ar[d] \ar[r, "\epsilon_S"] & S \ar[d] \\
      E(\Omega(T)) \ar[r, "\epsilon_T"] & T.
    \end{tikzcd}
  \]
  This is not hard to see; for example,
  suppose that $f\colon A \to B$ is a homomorphism of Boolean algebras.
  The map $f^*\colon E(B) \to E(A)$
  sends an ultrafilter $\omega$ on $B$ to the ultrafilter
  $f^*\omega = \{ a \in A : f(a) \in \omega\}$ on $A$.
  This is a continuous map,
  and the preimage of $U_a$ under this map
  is the set of ultrafilters
  $\{\omega : f^*\omega \in U_a \} = \{\omega : a \in f^*\omega \}
  = \{\omega : f(a) \in \omega \} = U_{f(a)}$,
  which shows that the left-hand square commutes.
  The right-hand square is similar; we leave it to the reader.
  
  The first part of the ``consequently'' statement
  is just the standard observation
  that the functors $E$ and $\Omega$ carry isomorphisms to isomorphisms
  (being equivalences of categories),
  and in fact yield isomorphisms of the categorical automorphism
  groups under consideration.

  Finally, we prove the ``moreover'' statement.
  Suppose that $K$ is compact and contained in the clopen set $U = U_a$.
  Notice that this containment means that for all $\omega \in K$,
  we have that $a \in \omega$.
  If $\varphi \in \aut(B)$ stabilizes $a$,
  we see that $\varphi^*(U_a) = U_{\varphi(a)} = U_a$,
  so in particular, we have that $\varphi^*(K) \subset U$---in other words,
  the identity neighborhood $V(K,U)$ in the compact--open topology
  on $\homeo(E) \cong \aut(B)$
  contains the identity neighborhood $\Stab(a)$ in the permutation topology.

  Conversely, suppose that $f$ is a homeomorphism of $E$
  contained in the identity neighborhood $V(U_a, U_a)$.
  Then $f^{-1}(U_a) = U_a$ by definition,
  so as a homomorphism of Boolean algebras,
  we have that $f^{-1} \colon \Omega(E) \to \Omega(E)$
  stabilizes $a$.
  That is, we see that the identity neighborhood $\Stab(a)$
  in the permutation topology
  contains the identity neighborhood $V(U_a,U_a)$ in the compact--open topology.
\end{proof}

\subsection{Non-Archimediean Polish groups}
Consider the case when $B$ is countable, $\aut(B)$ is non-Archimedean and Polish.  Since a Boolean algebra $B$ is, in particular, a poset,
we may consider the $1$-skeleton of its geometric realization,
a simplicial graph we will call $\Gamma_B$.
Vertices of $\Gamma_B$ are thus elements of $B$,
and we have a directed edge from $a$ to $b$
(which we assume distinct)
if $a \le b$ in the partial order.
In terms of the Boolean algebra,
this is the case when $a \vee b = b$.
When $B$ is countable, this is a countable graph
and we see immediately that the action of $\aut(B)$
on the vertex set of $\Gamma_B$ is faithful.

This action on a countable set yields a representation
$\aut(B) \to \Sym(\mathbb{N})$,
the group of permutations of a countable set.
Now, the group $\Sym(\mathbb{N})$ is \emph{Polish,}
meaning it has a countable dense subset
and is completely metrizable.

To see this, we claim that the (normal!)
subgroup of finitely supported permutations is dense.
Indeed, given $\Phi$ a permutation of $\mathbb{N}$
and $F \subset \mathbb{N}$ finite,
extend the action of $\Phi$ on $F$
to a bijection $\varphi_F$ of $F \cup \Phi(F)$
by for example setting $\varphi_F(x) = \Phi^{-1}(x)$ if $x \in \Phi(F) - F$.
Extend $\varphi_F$ to all of $\mathbb{N}$
by acting as the identity on the complement.
As we enlarge the set $F$,
the elements $\varphi_F^{-1}\Phi$ approach the identity.
Put another way,
every open neighborhood of $\Phi$ contains
an element of the form $\varphi_F$ in the permutation topology,
so $\Phi$ is in the closure of the finitely supported permutations of $\mathbb{N}$.

Define a metric on $\Sym(\mathbb{N})$ as follows:
if $\Phi$ and $\Psi$ are bijections of $\mathbb{N}$,
let $\delta(\Phi, \Psi) = 2^{-n}$,
where $n$ is the smallest element of $\mathbb{N}$
for which $\Phi(n) \ne \Psi(n)$.
Although this defines a left-invariant metric on $\Sym(\mathbb{N})$,
this metric is not complete~\cite{Cameron}.
Define instead $d(\Phi, \Psi) = \min\{\delta(\Phi,\Psi), \delta(\Phi^{-1},\Psi^{-1})\}$.
This metric is complete: if $\Phi_n$ is a Cauchy sequence,
one can show that $\Phi_n(k)$ and $\Phi_n^{-1}(k)$ stabilize for large $n$,
and therefore define a map
(which turns out to be a bijection)
$\Phi\colon \mathbb{N} \to \mathbb{N}$
by the rule that $\Phi(j) = \lim_{n \to \infty} \Phi_n(j)$.

Let us return to Boolean algebras.
Since the operations of meet and join
are recoverable from the poset structure
(for example $a \wedge b$ is the \emph{greatest lower bound} of $a$ and $b$),
if an element of $\aut(\Gamma_B)$ preserves directed edges,
we claim that actually it comes from an element of $\aut(B)$.
Indeed, since $0$ is the minimum and $1$ the maximum element of $B$,
we have that an element of $\aut(\Gamma_B)$
preserving directed edges preserves $0$, $1$,
$\wedge$ and $\vee$,
hence comes from an automorphism of $B$.
In general (i.e.\ for Boolean algebras satisfying $0 \ne 1$),
an element of $\aut(\Gamma_B)$
must either fix $0$ (and hence fix $1$) or \emph{swap them,}
so we have that $\aut(B)$ is an index-two subgroup of $\aut(\Gamma_B)$
(in general).
In the permutation topology on $\aut(\Gamma_B)$,
we see that $\aut(B) = \Stab(0)$ is open
(hence closed by general theory of topological groups).

Since $\Gamma_B$ is simplicial,
the edge relation on $\Gamma_B$
is described by a subset of the set of unordered pair of vertices of $\Gamma_B$.
The group $\Sym(\mathbb{N})$ acts continuously
on the vertex set of $\Gamma_B$ when given the discrete topology,
hence continuously on the set of unordered pairs.
An element of $\aut(\Gamma_B)$ is precisely an element
of $\Sym(\mathbb{N})$ that preserves the edge relation.
Like all subsets of the (discrete) set of unordered pairs,
the edge relation is closed.
Since $\Sym(\mathbb{N})$ acts continuously,
we see that for each unordered pair $[x,y]$,
the map $\Phi \mapsto [\Phi(x), \Phi(y)]$ is continuous,
hence the subgroup of $\Sym(\mathbb{N})$ preserving the edge relation
is closed in $\Sym(\mathbb{N})$.
Thus we see that $\aut(B)$ is a closed subgroup of $\Sym(\mathbb{N})$.

Closed subspaces of complete metric spaces are complete,
and subspaces of separable metric spaces are separable,
so we see that $\aut(B)$ is again a Polish group,
provided that $B$ is countable.

\subsection{Cut Complexes}

\begin{defn}
Let $B$ be a Boolean algebra
with associated Stone space $E = E(B)$.
A \emph{cut} is
(equivalently)
an unordered pair of elements $[a, \lnot a]$ of $B$
or a partition of $E$ into two disjoint clopen sets $U \sqcup V$.
A cut is \emph{non-peripheral}
if $U$ and $V$
each contain at least two points,
or, equivalently, if neither $a$ nor $\lnot a$ are $0$
or an \emph{atom,}
that is, an immediate successor of $0$ in the partial order.
\end{defn}

We see immediately that if $B$ admits non-peripheral cuts,
then the space $E$ has at least four points.

\begin{defn}
Two cuts $U \sqcup V$ and $U' \sqcup V'$ \emph{cross}
if each of the four pairwise intersections
$U \cap U'$,
$U \cap V'$,
$V \cap U'$
and $V \cap V'$
are nonempty.
In terms of the algebra,
this says that $[a,\lnot a]$ and $[b, \lnot b]$ cross
if none of the elements $a \wedge b$,
$a \wedge \lnot b$,
$\lnot a \wedge b$ or $\lnot a \wedge \lnot b$ are $0$.

If two cuts do not cross, we say that they are \emph{compatible.}

\end{defn}

Observe that peripheral cuts are compatible with \emph{every} cut,
and that if $B$ admits a pair of compatible non-peripheral cuts,
then $E$ has at least five points.

We come now to the main definition.
\begin{defn}
The \emph{complex of cuts} in $B$ (or $E$)
is the simplicial graph $\mathscr{C}(E)$
whose vertices are the non-peripheral cuts
in $B$ (or $E$),
with an edge between distinct cuts whenever they are compatible.
\end{defn}

The group $\aut(B)$ clearly acts on $\mathscr{C}(E)$
continuously.
Our main result is the following.
\begin{thm}\label{maintheorem}
  Suppose $E$ is second countable and has at least five points.
  The group $\aut(B) = \homeo(E)$ is isomorphic
  as a topological group to $\aut(\mathscr{C})$.
\end{thm}

As already mentioned,
the complex $\mathscr{C}(E)$
is empty when $E$ has three or fewer points.
When $E$ is a four-point set,
the complex $\mathscr{C}(E)$
is a set of three points (with no edges),
and the map $\homeo(E) \to \aut(\mathscr{C})$
is the exceptional map $S_4 \to S_3$
whose kernel is the Klein 4-group
generated by the products of disjoint transpositions.
When $E$ is a five-point set,
the graph $\mathscr{C}(E)$ is the Petersen graph
(cuts are in one-to-one correspondence with two-element
subsets of five points, and edges between cuts
correspond to disjointness of the two-element subsets),
which is well-known to have automorphism group $S_5$.

Recall that in the classification of infinite-type surfaces~\cite{Richards1963OnTC},
a fundamental invariant is the nested triple of Stone spaces $E_n(S)\subseteq E_g(S)\subseteq E(S)$,
where $E_n(S)$ is the space of ends accumulated by crosscaps
and $E_g(S)$ is the space of ends accumulated by genus.
Motivated by this example, we make the following definition.
\begin{defn}
  A \emph{Stone space system} is a properly nested collection of Stone spaces
  $E_n\subseteq E_{n-1}\subseteq \cdots \subseteq E_2\subseteq E_1$ for some $n \ge 2$.
  A \emph{cut} of a Stone space system is a cut of $E_1$.
  A \emph{homeomorphism} of Stone space systems is a homeomorphism of $E_1$
  which preserves each $E_k$ for $2\leq k\leq n$.
  The number $n$ is called the \emph{length} of the Stone space system.
\end{defn}

When is a cut of a Stone space system ``non-peripheral''?
There are two definitions we consider.

\begin{defn}
  Let $E=(E_1,\ldots, E_n)$ be a Stone space sytem.
  A cut of $E$ is \emph{weakly non-peripheral}
  if each part of the partition either contains at least two points or at least one point of $E_2$.
  The \emph{weak complex of cuts} $\mathscr{C}_w(E)$ is the complex whose vertices are weakly non-peripheral cuts
  and whose edges correspond to compatibility.
\end{defn}

Intuitively, we think of the surface $S$ obtained by removing $E_1$ from a sphere
and gluing infinite rays of handles or crosscaps at each point of $E_2$.
Then a weakly non-peripheral cut of $E$ is roughly analogous to a simple closed curve of $S$
which is not contractable or homotopic to a puncture.

We can immediately see that the main result does \emph{not}
extend to the weak cut complex for Stone space systems with length at least 3.
Consider the Stone space system $E=(E_1,E_2,E_3)$ where
\[E_1=\{a,b,c,d,e\}\]
\[E_2=\{a,b\}\]
\[E_3=\{a\}.\]

The cuts $A= \{a\} \sqcup \{b,c,d,e\}$ and $B=\{b\} \sqcup \{a,c,d,e\}$ are both weakly non-peripheral.
Both of these cuts are compatible with all other cuts.
Hence, there exists an automorphism of the weak complex of cuts
$\phi\colon \mathscr{C}_w(E)\to \mathscr{C}_w(E)$ such that $\phi(A)=B$.
But a homeomorphism of $E_1$ which induces $\phi$ must take $a$ to $b$,
and hence fails to be a homeomorphism of the Stone space system.

In fact, the result fails even for the weak complex of \emph{pairs} of Stone spaces.
\begin{lem}
  Let $E_1$ be a Stone space with at least five points, let $k\in E_1$ be an isolated point,
  and let $E$ be the Stone space system $(E_1,\{k\})$.
  Then $\aut(\mathscr{C}_w(E))\cong \aut(\mathscr{C}(E_1))$.
\end{lem}

In particular, suppose $E_1$ is a discrete finite set with $n\geq 5$ points.
Then $\aut(\mathscr{C}_w(E))\cong \aut(\mathscr{C}(E_1))\cong S_n$,
but the homeomorphism group of $E$ is isomorphic to $S_{n-1}$.

\begin{proof}
  Let $\kappa$ be the cut $\{k\} \sqcup E_1 - \{k\}$.
  Then $\kappa$ is weakly-non-peripheral by definition.
  Moreover, all weakly-non-peripheral cuts except for $\kappa$  are non-peripheral as cuts in $E_1$.
  Hence, $\mathscr{C}(E_1)$ is the full subgraph of $\mathscr{C}_w(E)$ obtained by removing the vertex $\kappa$.
  
  Every cut is compatible with $\kappa$, hence $\kappa$ is adjacent to every other vertex of $\mathscr{C}_w(E)$.
  No other vertices have this property: thus $\kappa$ is fixed by the automorphism group of $\mathscr{C}_w(E)$,
  so removing $\kappa$ does not change the automorphism group.
\end{proof}

\begin{defn}
  Let $E=(E_1,\ldots, E_n)$ be a Stone space system.
  A cut of $E$ is \emph{strongly non-peripheral} if each part of the partition contains at least two elements of $E_n$.
  The \emph{strong cut-complex} $\mathscr{C}_s(E)$ is the full subgraph of $\mathscr{C}(E_1)$ comprising the strongly non-peripheral vertices.
\end{defn}

Once again, the main result does not generalize to strong cut complexes for Stone space systems of length at least 3.
Consider the system $E=(E_1, E_2, E_3)$ with
\[E_1=\{a,b,c,d,e,f,g\}\]
\[E_2=\{a,b,c,d,e,f\}\]
\[E_3=\{a,b,c,d,e\}.\]

Consider the homeomorphism $\phi\colon E_1\to E_1$ given by
\[\phi(f)=g,\]
\[\phi(g)=f,\]
\[\phi(x)=x, \quad\mathrm{for }\quad x\in E_3.\]
Note that $\phi$ is \emph{not} a homeomorphism of the Stone space system,
since it swaps an element of $E_2$ with an element of $E_1\backslash E_2$.
However, $\phi$ induces an automorphism on the strong cut-complex.
To see this, note that the strongly non-peripheral property depends only on the elements of $E_3$,
which are all fixed pointwise by $\phi$.
Hence, a cut $X$ is strongly non-peripheral if and only if $\phi(X)$ is strongly non-peripheral.

\begin{question}\rm
  If $E=(E_1,E_2)$ is a Stone space pair,
  do we have that $\homeo(E)\cong \aut(\mathscr{C}_s(E))$?
\end{question}

\begin{lem}
  The graph $\mathscr{C}(E)$ is connected
  if $E$ has at least five points.
  Moreover if $E$ is finite,
  then $\mathscr{C}(E)$ has diameter at most four,
  while if $E$ is infinite,
  $\mathscr{C}(E)$ has diameter two.
\end{lem}

\begin{proof}
  Suppose that $E$ is finite but has at least five points.
  Any choice $U_\gamma$ of two points from $E$
  determines a cut $U_\gamma \sqcup (E - U_\gamma)$,
  and any cut $\gamma$ is compatible with a cut $\gamma'$,
  one of whose sets $U_{\gamma'}$ has size two.
  So beginning with cuts $\gamma$ and $\eta$,
  replace them with compatible cuts $\gamma'$ and $\eta'$
  respectively whose smaller subset $U_{\gamma'}$ and $U_{\eta'}$ has size two.
  If $\gamma'$ and $\eta'$ are compatible (or equal),
  then we are done.
  If not, then $U_{\gamma'} \cup U_{\eta'}$ has size three,
  and since $E$ has at least five elements,
  this determines a cut compatible with both $\gamma'$ and $\eta'$.
  This proves that $\mathscr{C}(E)$ has diameter at most four when $E$ is finite.

  Supposing instead that $E$ is infinite,
  note that at least one of the subsets determined by a cut $\gamma$
  is infinite.
  If $\gamma = U \sqcup V$ and $\eta = U' \sqcup V'$
  are cuts that cross,
  then each of the four intersections is nonempty
  and at least one must be infinite.
  That intersection determines a cut,
  since its complement contains at least three points
  (one for each of the remaining pairwise intersections),
  and this cut is compatible with both $\gamma$ and $\eta$.
  This proves that $\mathscr{C}(E)$ has diameter two.
\end{proof}

\subsection{Cantor sets and subcomplexes}
In the study of infinite-type surfaces, the curve complex has diameter two.
However, there are natural full subgraphs which have infinite diameter~\cite{AFP2017}.
One could ask whether such a subcomplex exists for the cut complex.
The next lemma shows that we cannot hope to get such a subcomplex for arbitrary Stone spaces.

\begin{lem}
  Let $K$ be a Cantor set, and let $G$ be a nonempty,
  $\homeo(K)$-invariant full subgraph of $\mathscr{C}(K)$.
  Then $G$ has diameter 2.
\end{lem}

\begin{proof}
  First, we show $G$ has diameter at least two.
  Since $G$ is non-empty, there exists a vertex $x\in V(G)$ corresponding to the non-peripheral cut $A\sqcup (K\backslash A)$ of $K$.
  Since $G$ is $\homeo(K)$-invariant, $G$ contains $f(x)$ for all homeomorphisms $f\in \homeo(K)$.

  $A$ is a nonempty proper clopen subset of a Cantor set, so $A$ is also homeomorphic to $K$.
  If we represent $K$ as the set of all infinite binary sequences,
  we may assume without loss of generality that $A$ is the set of binary sequences whose first term is 0.

  Let $f\colon K\to K$ be the homeomorphism that sends the sequence $(a_1,a_2,a_3,a_4,\ldots)$ to  the sequence $(a_2,a_1,a_3,a_4,\ldots)$.
  Observe that $x$ and $f(x)$ cross: $A\cap f(A)$ is the set of sequences beginning with two 0s;
  $A\cap (K\backslash f(A))$ is the set of sequences beginning with a zero followed by a one.
  Hence, $d_G(x,f(x))\geq 2$.

  Now we show $G$ has diameter at most two.
  Suppose $x=A\sqcup (K\backslash A)$ and $y=B\sqcup (K\backslash B)$ are two vertices of $G$.
  If $x$ and $y$ are adjacent, then we are done.
  Otherwise, $A\cap B$ is a nonempty clopen proper subset of $K$
  different from $A$ or $B$,
  and hence there exists a homeomorphism $f\colon K\to K$ such that $f(A)=A\cap B$.
  Thus $f(x)=(A\cap B)\sqcup (K\backslash (A\cap B))$ is also a vertex of $G$, and $f(x)$ is adjacent to both $x$ and $y$.
\end{proof}

\section{Pants Decompositions}

Our proof of \Cref{maintheorem} is motivated
by the analogy between $E$
and the end spaces of infinite-type surfaces,
in particular the papers~\cite{Valdez1,Valdez2}.
Motivated by pants decompositions of surfaces, we introduce the following definition.

\begin{defn}\rm
Let $\{\gamma_n\}$ be a countable collection of \emph{non-peripheral} cuts.
We say that $\Gamma = \{\gamma_n\}$
is a \emph{pants decomposition} of $E$
if it satisfies the following conditions.
\begin{enumerate}
\item Any two cuts $\gamma_i$ and $\gamma_j$ in $\Gamma$ are compatible.
\item If $\gamma$ is a cut \emph{not} in $\Gamma$ but in $\mathscr{C}(E)$,
  then there exists some cut $\gamma_j \in \Gamma$
  which crosses $\gamma$.
\item Moreover the collection of $j$ such that $\gamma$ crosses $\gamma_j$
  is finite.
\end{enumerate}

\end{defn}

Since $\mathscr{C}(E)$ is a simplicial graph,
we may think of it as a (flag) simplicial complex
by adding a simplex whenever its $1$-skeleton is present.
The first two properties say that $\Gamma$
is the set of vertices of a maximal simplex in $\mathscr{C}(E)$.

If $E$ is finite, the third condition is vacuous,
since $\mathscr{C}(E)$ is a finite graph.
If $E$ is infinite, we have the following result.

\begin{lem}
  Suppose that $E$ is second countable (or equivalently that $B$ is countable)
  and that $E$ has at least four elements.
  Pants decompositions of $E$ exist.
\end{lem}

\begin{proof}
  We begin with a particular example.
  Recall the construction of the standard middle-thirds Cantor set $\mathscr{C}$:
  one begins with the unit interval $[0,1]$ and then
  at the $k$\textsuperscript{th} level of the construction for $k \ge 1$,
  a number $0 \le x \le 1$ remains in the set
  provided none of the first $k$ digits after the decimal of $x$ expressed as a number in base 3
  are $1$. (We represent $1$ as $0.\overline{2}$.)
  There are $2^k$ sequences of length $k$ drawn from the alphabet $\{0,2\}$;
  each choice of a sequence $s_1\ldots s_k$ determines a cut
  $\gamma_{s_1\ldots s_k} = U_{s_1\ldots s_k} \sqcup V$,
  where $U_{s_1\ldots s_k}$ is the clopen subset of $\mathscr{C}$ comprising those $x \in \mathscr{C}$
  whose first $k$ digits expressed as a number in base 3 are $0.s_1\ldots s_k$,
  and $V$ is its complement in $\mathscr{C}$.

  The collection of finite sequences in the alphabet $\{0,2\}$ is countable,
  and cuts in $\mathscr{C}$ corresponding to distinct finite sequences are compatible;
  we claim that the collection $\Gamma = \{\gamma_{\vec s} : \vec s = s_1\ldots s_k,\ s_i \in \{0,2\}\}$
  is a pants decomposition of $\mathscr{C}$.
  Let $\gamma = U \sqcup V$ be a cut not in $\Gamma$.
  Since the sets $U_{s_1\ldots s_k}$ form a neighborhood basis for $\mathscr{C}$,
  using compactness we may write $U$ and $V$ each as a disjoint union of finitely many of these sets,
  say $U = U_1 \sqcup \cdots \sqcup U_n$ and $V = V_1 \sqcup \cdots V_m$.
  Any cut $\gamma_{\vec s}$ for which the sequence $\vec s = s_1\ldots s_k$ has a prefix equal
  to the string associated to some $U_i$ or $V_j$ will be compatible with $\gamma$.
  Since there is a bound to the length of strings appearing in the lists associated to $U_1,\ldots,U_n$
  and $V_1,\ldots,V_m$ and $U \sqcup V$ is a cut so that every $\vec s$ longer than that bound has
  such a prefix,
  there are only finitely many $\gamma_{\vec s}$ which may cross $\gamma$.
  Thus $\Gamma$ is a pants decomposition of the Cantor set.

  Now suppose that $E$ is a second-countable Stone space.
  By Stone duality, $E$ is (homeomorphic to) a closed subset of $\mathscr{C}$,
  and each cut $\gamma \in \Gamma$ restricts to a cut $\bar\gamma = \bar U \sqcup \bar V$,
  where $\bar U = E \cap U$ and $\bar v = E \cap V$.
  Now, these cuts may not all be non-peripheral.
  We restrict ourselves to
  $\Gamma_E = \{ \bar\gamma : \bar\gamma \text{ is nonperipheral and } \gamma \in \Gamma \}$.
  It is clear that this collection of cuts is countable and pairwise compatible.
  We claim that it is a pants decomposition of $E$.

  Now, by the definition of the subspace topology,
  \emph{every} cut $\bar U \sqcup \bar V$ in $E$ is the restriction $\bar\gamma$
  of a cut $\gamma = U \sqcup V$ in $\mathscr{C}$.
  A nonperipheral cut in $E$ will belong to $\Gamma_E$
  if and only if it is equal to $\bar\gamma$ for some cut $\gamma$ in $\Gamma$.
  (It may \emph{also} be equal to $\bar\gamma'$ where $\gamma'$ is not in $\Gamma$.)
  Therefore if a cut does not belong to $\Gamma_E$,
  then for every expression of that cut as $\bar\gamma$,
  we have that the cut $\gamma$ is not in $\Gamma$.
  Fixing such a cut $\gamma = U \sqcup V$,
  write $U = U_1 \sqcup \cdots U_n$,
  where each $U_i$ is $U_{\vec {s}_i}$ for some finite sequence $\vec {s}_i$.
  Since $\bar\gamma$ is nonperipheral,
  at least two of these $U_i$ contain points of $E$, say $\vec {s}_i$ and $\vec {s}_j$
  and since $\gamma$ cannot be chosen from $\Gamma$,
  if ${\vec s}$ is a common prefix with both $\vec s_i$ and $\vec s_j$,
  then we may assume without loss of generality that $U_{\vec s}$ contains a point of $\bar V$.
  By paying similar attention to $V$,
  we may choose such a $\vec s$ such that the complement of $U_{\vec s}$
  contains at least two points of $\bar V$,
  so that $\bar\gamma_{\vec s}$ is nonperipheral as a cut of $E$ hence belongs to $\Gamma_E$
  and crosses $\bar\gamma$.
  This establishes the second property of pants decompositions.
  For the third, note that if a cut crosses $\bar\gamma_{\vec s} \in \Gamma_E$,
  then \emph{every} representation of that cut as $\bar\gamma$
  for $\gamma$ a cut in $\mathscr{C}$ will cross $\gamma_{\vec s}$,
  so since $\Gamma$ is a pants decomposition,
  the set of $\bar\gamma_{\vec s}$ crossing a given cut $\bar\gamma$ not in $\Gamma_E$ is finite,
  since it is contained in the finite set of cuts in $\Gamma$ crossing $\gamma$,
  even though \emph{a priori} infinitely many choices of such a $\gamma$ are possible.
\end{proof}

\begin{defn}\rm
If $\Gamma$ is a pants decomposition, two cuts $\gamma_i$ and $\gamma_j$ of $\Gamma$ are \emph{adjacent} if there exists a cut $\gamma$ in $\mathscr{C}(E)$
such that $\gamma$ crosses $\gamma_i$ and $\gamma_j$
but no other cut in $\Gamma$.

The \emph{adjacency graph} of $\Gamma$,
written $A(\Gamma)$, is the simplicial graph with vertex set $\Gamma$,
with an edge between two vertices whenever the corresponding
cuts $\gamma_i$ and $\gamma_j$
are adjacent in $\Gamma$.

\end{defn}
Because adjacency is preserved by isomorphisms,
we have the following result.

\begin{lem}
  Suppose $\Phi\colon \mathscr{C}(E) \to \mathscr{C}(E')$
  is an isomorphism.
  Then for any pants decomposition $\Gamma$
  of $E$,
  $\Phi(\Gamma)$ is a pants decomposition of $E'$
  and $\Phi$ induces an isomorphism of graphs $A(\Gamma) \to A(\Phi(\Gamma))$.
  \hfill\qedsymbol{}
\end{lem}

\begin{defn}\rm
A cut in $\mathscr{C}(E)$ is said to be \emph{outermost}
if at least one component of the corresponding partition of $E$
contains exactly two points.
\end{defn}

\begin{lem}
  Suppose that $E$ has at least seven elements.
  A cut $\gamma$ in $\mathscr{C}(E)$
  is outermost if and only if for each pants decomposition
  $\Gamma$ of $E$ containing $\gamma$,
  the vertex $\gamma$ has at most valence two in $A(\Gamma)$.
\end{lem}

Notice that the lemma is not true for $E$ having six points,
the reason being that maximal simplices in $\mathscr{C}(E)$
are of dimension two in this case,
each maximal simplex is a pants decomposition $\Gamma$,
and the graph $A(\Gamma)$ is connected.
Each such graph $A(\Gamma)$ has three vertices,
so each vertex of the graph has at most valence two,
regardless of whether it is outermost,
and there exist non-outermost cuts in $E$ in this case.
However, the corollary that we will derive from it,
namely that isomorphisms
of complexes of cuts
send outermost cuts to outermost cuts,
is still true for $E$
having between four and six elements;
this is because either \emph{every} cut is outermost
(for four and five)
or outermost cuts are distinguished in $\mathscr{C}(E)$
from others by their valence in $\mathscr{C}(E)$
(for six).

\begin{proof}
  Suppose that $\gamma = U \sqcup V$
  is a cut in $E$ which is not outermost.
  Then since $E$ has at least seven elements,
  we have that $U$ has at least three and $V$ at least four
  (up to swapping them).
  In particular $U$ contains a clopen subset $U'$
  with at least two elements and $V$ contains disjoint clopen subsets
  $V'$ and $V''$ whose union is $V$
  and for which each subset contains at least two elements.
  These sets yield cuts $\gamma_1$, $\gamma_2$ and $\gamma_3$.
  We claim that $\gamma$ is adjacent to each of these cuts
  in some pants decomposition containing all four.
  Indeed, by definition,
  we may further divide $U'$ into clopen subsets $U'^+$ and $U'^-$
  with at least one element,
  and one sees directly that $U'^+ \sqcup V'$ is half a cut which
  crosses $\gamma$
  and $\gamma_1$,
  and that by using the clopen complement of $U'$
  in $U$ together with halves of $V'$ and $V''$,
  we can construct cuts that cross $\gamma$
  and $\gamma_2$ and cuts that cross $\gamma$ and $\gamma_3$.
  By choosing these halves to be sides of cuts
  in our pants decomposition if necessary,
  we can show that these cuts cross \emph{only}
  those elements,
  showing that $\gamma$ is adjacent to $\gamma_1$, $\gamma_2$ and $\gamma_3$
  as desired.

  Now suppose that $\gamma = U \sqcup V$
  is outermost with $U$ containing two elements.
  Since $E$ contains at least seven elements,
  in any pants decomposition $\Gamma$ of $E$
  containing $\gamma$,
  every cut $\gamma_n \in \Gamma$
  has a side $V_n$ contained in $V$.
  Moreover, we have that
  the poset of clopen subsets $V_n$ properly contained in $V$
  has at most two maximal elements with respect to inclusion.
  Indeed, if it had three, say $V_1$, $V_2$ and $V_3$,
  then $V_1 \sqcup V_2$ would determine a cut
  compatible with every cut in $\Gamma$,
  in contradiction to the assumption that $\Gamma$ is a pants decomposition.
  Arguing as before,
  we can show that $\gamma$ is adjacent to the cuts
  determined by $V_1$ and $V_2$
  and no others.
\end{proof}

\begin{defn}
A pair of cuts $\gamma$ and $\eta$ are said to
form a \emph{peripheral pair}
if they are compatible and if,
thinking of them as partitions $U\sqcup V$ and $U'\sqcup V'$,
one of the sets $U\cap U'$, $U \cap V'$, $V \cap U'$ or $V \cap V'$ is a singleton.
Notice that since $\gamma$ and $\eta$ are compatible,
after relabeling, we may assume that $U \cap U' = \varnothing$,
from which it follows that the singleton must be $V \cap V'$.
\end{defn}

Observe that in any pants decomposition $\Gamma$ of $E$
containing the peripheral pair $\gamma$ and $\eta$,
these cuts $\gamma$ and $\eta$ are adjacent in $\Gamma$.
The simplest example happens when $E$ is a set of five points, say
$E = \{ \star, a_1, a_2, b_1, b_2 \}$.
One such peripheral pair has $\gamma = \{ a_1, a_2 \} \sqcup \{ \star, b_1, b_2 \}$ and $\eta = \{ b_1, b_2 \} \sqcup \{ \star, a_1, a_2 \}$.
One requisite cut demonstrating adjacency is
$\kappa = \{a_1, b_1 \} \sqcup \{ \star, a_2, b_2 \}$.

We have the following pair of lemmas.

\begin{lem}
  Suppose that $\gamma$ is an outermost cut
  and that $\gamma$ and $\eta$ form a peripheral pair.
  Then in any pants decomposition $\Gamma$ containing $\gamma$ and $\eta$,
  we have that $\gamma$ has valence one in $A(\Gamma)$.
\end{lem}

\begin{proof}
  Indeed, suppose $\gamma = U \sqcup V$
  is adjacent to $\eta = U' \sqcup V'$,
  where $U = \{ y, z \}$ has size two
  and $\gamma$ and $\eta$ form a peripheral pair.
  We may suppose that $U' = U \sqcup \{x\}$.
  Suppose $\gamma$ is adjacent to another cut 
  $\delta = U'' \sqcup V''$
  in $\Gamma$.
  Then without loss of generality $U' \subset U''$.
  Both sides of any cut $\kappa$ witnessing the adjacency of $\gamma$ and $\delta$
  must have nonempty intersection with $U$ and with $V$.
  The only nonperipheral choices for such a $\kappa$
  force one side of $\kappa$ to be either
  $\{ x, y \}$ or $\{ x, z \}$,
  since any other choice produces a cut which crosses $\eta$.
  But these choices produce cuts $\kappa$ which fail to cross $\delta$.
  This contradiction shows that $\eta$
  is the only cut in $\Gamma$ to which $\gamma$ is adjacent.
\end{proof}

\begin{defn}
For a simplicial graph like $\mathscr{C}(E)$,
if $\gamma$ is a vertex,
we define
the \emph{link} $L(\gamma)$
of $\gamma$ to be the \emph{full} or \emph{induced}
subgraph comprising those vertices of $\mathscr{C}(E)$
which are adjacent to $\gamma$,
with an edge between two vertices of the link when the corresponding vertices
of $\mathscr{C}(E)$ are connected by an edge.
\end{defn}

\begin{defn}
Given a simplicial graph $L(\gamma)$,
its \emph{opposite graph} ${L(\gamma)}^\perp$
is the graph on the same vertex set but with an edge between vertices
when there is \textbf{no} edge in $L(\gamma)$.
\end{defn}

Observe that if $\gamma = U \sqcup V$ is a cut in $E$,
we may form two new Stone spaces,
namely $U \sqcup \{V\}$ and $V \sqcup \{U\}$
by alternately collapsing all points in $V$
to a singleton or collapsing all points in $U$.
Observe that every non-peripheral cut in $\mathscr{C}(E)$
compatible with $\gamma$
yields a non-peripheral cut in one of these Stone spaces
but not both,
and conversely, every non-peripheral cut in one of these Stone
spaces yields a cut in $E$ compatible with $\gamma$.
Indeed, it is not hard to see that $L(\gamma)$
is isomorphic to the \emph{join} of the graphs
$\mathscr{C}(U \sqcup\{V\})$ and $\mathscr{C}(V \sqcup\{U\})$.
If $\gamma$ is not outermost,
then both of these graphs are nonempty
and we see that ${L(\gamma)}^\perp$ has two components.

Here is some intuition.
Imagine embedding $E$ into $S^2$
(this is, of course, possible if and only if $E$ is second countable)
and realizing a cut $\gamma$ as a Jordan curve (also called $\gamma$)
in $S^2 - E$
which separates points of $U$ from those of $V$.
The cut $\gamma$ is peripheral if and only if one of the two components
of $S^2 - (E \cup \gamma)$ contains exactly one point of $E$
(i.e. is homeomorphic to an open annulus).
Every cut in $\mathscr{C}(E)$ compatible with yet distinct from $\gamma$
may be drawn on $S^2$ disjoint from $\gamma$;
thus it is contained in one of the two components of $S^2 - (E \cup \gamma)$.
One such component is homeomorphic to the complement in $S^2$ of the Stone space $U \sqcup \{V\}$,
while the other is homeomorphic to the complement of $V \sqcup \{U\}$.
If $\eta$ is a nonperipheral cut in $E$ compatible with $\gamma$,
it is a nonperipheral cut in one component of the complement and does not appear in the other.

Conversely, every nonperipheral cut in $U \sqcup \{V\}$ or $V \sqcup \{U\}$
yields, under the homeomorphism of complements above, a nonperipheral cut in $E$ compatible with $\gamma$.
Nonperipheral cuts, recall, exist in every Stone space with at least four points,
so $\mathscr{C}(U \sqcup \{V\})$, say, will be empty only when $U$ is a two-point set.
Since one forms the join $\Gamma\star\Gamma'$
of two graphs $\Gamma$ and $\Gamma'$ by beginning with the disjoint union
and connecting every vertex of $\Gamma$ to every vertex of $\Gamma'$,
we have that $(\Gamma\star\Gamma')^\perp = \Gamma^\perp \sqcup \Gamma'^\perp$,
which has at least two connected components when $\Gamma$ and $\Gamma'$ are nonempty.
Indeed, in the case of $L(\gamma)^\perp$, this graph has exactly two components, since
in every Stone space with at least four points, for any two compatible cuts
it's easy to see there is a third which crosses both of them.

\begin{lem}\label{opposite graph}
  Suppose that $\gamma$ and $\eta$ are compatible cuts
  but that neither $\gamma$ nor $\eta$ is outermost.
  Then we have that $\gamma$ and $\eta$ form a peripheral pair
  if and only if the opposite graph of the full subgraph
  on the intersection of their links,
  that is,
  ${(L(\gamma) \cap L(\eta))}^\perp$, has two connected components.
\end{lem}

\begin{proof}
  From the foregoing discussion,
  we may think of a cut $\delta$ in $L(\gamma)$
  as coming from a cut in one of the Stone spaces $U\sqcup\{V\}$
  or $V \sqcup\{U\}$, where $\gamma = U\sqcup V$.
  If $\delta$ also belongs to $L(\eta)$,
  then $\delta$ also corresponds to a cut in $U'\sqcup\{V'\}$
  or in $V' \sqcup\{U'\}$.
  Since $\gamma$ and $\eta$ are compatible,
  we may suppose that $U \cap U' = \varnothing$,
  or in other words, that there are \emph{no} cuts in $L(\gamma) \cap L(\eta)$
  corresponding to cuts in both $U\sqcup\{V\}$ and $U' \sqcup \{V'\}$.
  If $\gamma$ and $\eta$ form a peripheral pair,
  if follows that $V \cap V'$ is a one-point set,
  and that there is again no cut $\delta$ which is distinct from
  $\gamma$ and $\eta$ but which corresponds
  to a cut in both $V\sqcup\{U\}$ and $V'\sqcup\{U'\}$,
  while if $\gamma$ and $\eta$ do \emph{not}
  form a peripheral pair,
  there \emph{is} such a cut $\delta$.

  The intuition of Jordan curves on $S^2$ is helpful again:
  \emph{a priori} ${(L(\gamma) \cap L(\eta))}^\perp$ has a maximum of three components,
  one for each connected component of $S^2 - (E \cup \gamma \cup \eta)$
  when $\gamma$ and $\eta$ are compatible.
  If neither $\gamma$ nor $\eta$ is outermost, both $U \sqcup \{V\}$ and $U' \sqcup \{V'\}$
  admit nonperipheral cuts, so there are at least two components in ${(L(\gamma) \cap L(\eta))}^\perp$.
  Assuming that $\gamma = U \sqcup V$, that $\eta = U' \sqcup V'$, and that $U \cap U' = \varnothing$,
  cuts in the third complementary component of $S^2 - (E \cup \gamma \cup \eta)$
  correspond to cuts in the Stone space $(V \cap V') \sqcup \{U, U'\}$.
  This space has nonperipheral cuts
  precisely when $\gamma$ and $\eta$ do \emph{not} form a peripheral pair.
  Therefore ${(L(\gamma) \cap L(\eta))}^\perp$ has two connected components if and only if
  the non-outermost cuts $\gamma$ and $\eta$ form a peripheral pair.
\end{proof}

The foregoing lemmas have the following corollary.

\begin{cor}
  If $\Phi\colon \mathscr{C}(E) \to \mathscr{C}(E')$
  is an isomorphism of complexes of cuts,
  then $\Phi$ sends peripheral pairs to peripheral pairs.
\end{cor}

We need a slight strengthening of \Cref{opposite graph}.
If the cuts $\gamma_1,\ldots,\gamma_k$ are merely pairwise compatible,
rather than a pants decomposition
notice that we may still define an adjacency graph for these cuts.
\begin{lem}\label{spheres before spheres}
  Suppose $E$ is second countable.
  Let $\gamma_1, \ldots, \gamma_k$ be pairwise compatible cuts.
  The graph ${(L(\gamma_1)\cap\cdots\cap L(\gamma_k))}^\perp$ has at most $k+1$ components.
  Moreover, if it has exactly $k+1$ components,
  then the following statements hold:
  \begin{enumerate}
  \item No cut $\gamma_i$ is outermost.
  \item No pair $\gamma_i$, $\gamma_j$ is peripheral.
  \item For every triple $\gamma_i = U_i \sqcup V_i$, $\gamma_j = U_j \sqcup V_j$, $\gamma_k = U_k \sqcup V_k$,
    if the adjacency graph $A(\{\gamma_i, \gamma_j, \gamma_k\})$ is a triangle,
 then we may write each cut so that $U_i \cap U_j = U_i \cap U_k = U_j \cap U_k = \varnothing$,
    and moreover the triple intersection $V_i \cap V_j \cap V_k$ is nonempty.
  \end{enumerate}
\end{lem}

\begin{proof} 
  Although the statement is likely true for Stone spaces which are not second countable,
  %% Is it true?
  the proof is again more straightforward if one considers an embedding of $E$ in $S^2$.
  Then each of the cuts $\gamma$ in $\{\gamma_1,\ldots,\gamma_k\}$ may be represented by a Jordan curve
  in $S^2$. Compatibility says these curves may be drawn disjoint from each other,
  and the complement of $S^2$ minus these $k$ curves has $k+1$ components.

  Perhaps the easiest way to see this is to form the graph dual to this system of Jordan curves on $S^2$.
  That is, place a vertex for each complementary component and connect components via an edge when they border the same curve.
  This graph is a tree, since $S^2$ is simply connected
  and any loop in the dual graph would be homotopically nontrivial in the surface.
  It is a standard fact that a tree with $k$ edges has $k + 1$ vertices.

  Similar to the case of one or two cuts, the closure of each component of
  $S^2 - (E \cup \gamma_1 \cup \cdots \cup \gamma_k)$
  will contain an essential simple closed curve
  (which corresponds to a non-peripheral cut)
  provided it is not homeomorphic to a sphere with $n$ punctures and $k$ boundary components
  with $n + k \le 3$.

  Translating this into the language of cuts,
  we see that if each complementary surface contains an essential simple closed curve,
  then $\gamma_i$ cannot be outermost (for then one component of $(S^2 - E) - \gamma_i$
  would be a disc with two punctures),
  no pair $\gamma_i$, $\gamma_j$ can be peripheral
  (this would yield a component which is a sphere with one puncture and two boundary components)
  and the third condition also holds:
  the adjacency graph is a triangle exactly when one component of $S^2 - \{\gamma_i,\gamma_j,\gamma_k\}$
  is a sphere with three boundary components;
  that sphere corresponds to $V_i \cap V_j \cap V_k$,
  and is therefore a sphere with three boundary components in $(S^2 - E) - \{\gamma_i,\gamma_j,\gamma_k\}$
  when this triple intersection is empty.
\end{proof}

\section{Realizing Automorphisms}

We prove one piece of \Cref{maintheorem}
in the following slightly stronger form.

\begin{thm}\label{isomorphisms are geometric}
  Suppose $\Phi\colon \mathscr{C}(E) \to \mathscr{C}(E')$
  is an isomorphism,
  where the spaces $E$ and $E'$ have at least four points
  and are second countable.
  Then there exists a homeomorphism $f\colon E \to E'$
  inducing the isomorphism $\Phi$.
\end{thm}

For use in the proof, we introduce the following definition.
\begin{defn}
A \emph{sphere in $E$ with $n$ punctures and $k$ boundary components:}
is a collection of $k$ pairwise-compatible non-peripheral cuts $\gamma_1,\ldots,\gamma_k$ in $E$,
thought of as partitions $U_1 \sqcup V_1$ through $U_k \sqcup V_k$
such that each $U_i \cap U_j$ for $i$ and $j$ distinct is empty
and the total intersection $V_1 \cap \cdots \cap V_k$
is a finite set of size $n \ge 0$.
The cuts $\gamma_1,\ldots,\gamma_k$
are said to be the \emph{boundary components}
of the sphere $S$,
and the points of the total intersection are said to be \emph{punctures.}
A cut $\gamma$ distinct from each of the $\gamma_i$
is said to be in the \emph{interior} of such a sphere $S$
is whenever a side of $\gamma$ contains a point of $U_i$,
it contains $U_i$ completely.
Said another way, in this situation,
one of the components of ${(L(\gamma_1) \cap \cdots \cap L(\gamma_k))}^\perp$
corresponds to nonperipheral cuts in the Stone space
$(V_1 \cap \cdots \cap V_k) \sqcup \{ U_1, \ldots, U_k \}$;
a cut $\gamma$ is interior provided it corresponds to a nonperipheral cut in this Stone space.
\end{defn}

Observe that if $\gamma$ is a cut in the interior of a sphere
with $n$ punctures and $k$ boundary components,
then $\gamma$ may be thought of as dividing $S$
into two spheres $S'$ and $S''$
whose total number of punctures is $n$
and whose total number of boundary components is $k + 2$.

\begin{defn}
A \emph{principal spherical exhaustion} of $E$
is a collection of spheres $S_1, S_2, \ldots$
satisfying the following conditions.
\begin{enumerate}
\item (Increasing)
  Each boundary component of $S_i$ is an interior cut in $S_{i+1}$.
\item (Exhaustion)
  Each cut $\gamma$ in $E$ is interior to some $S_n$.
\item (Complexity)
  Each sphere formed from a boundary component of $S_i$
  together with those from $S_{i+1}$,
  supposing it has $n$ punctures and $k$ boundary components,
  satisfies $n + k \ge 5$.
\item (Infinite Complement)
  For each boundary component $\gamma = U \sqcup V$
  of $S_n$ where $U$ is chosen in the notation as above,
  the clopen set $U$ is infinite.
\end{enumerate}
\end{defn}

By beginning with a pants decomposition of $E$
and forgetting certain pants curves,
one sees that principal spherical exhaustions exist whenever $E$
is infinite and second countable.
For each sphere $S_i$ in the exhaustion,
let $\bar S_i$ be a finite set with $n + k$ points,
one for each puncture and boundary component of $S_i$.
Then observe that there exist continuous maps $E \to \bar{S}_i$ and $\bar{S}_{i+1} \to \bar{S}_i$:
for the former, one collapses each clopen set $U$ for each boundary component in $S_i$ to a single point;
continuity follows because $U$ is clopen.
The same argument works for the latter case by thinking of each boundary component of $S_i$
as yielding a clopen set in $\bar {S}_{i+1}$.
We have that for each $i$, the following triangle commutes
\[
  \begin{tikzcd}
    & E \ar[ld] \ar[rd] & \\
    \bar{S}_{i+1} \ar[rr] & & \bar{S}_i. \\
  \end{tikzcd}
\]

The maps $\bar{S}_{i+1} \to \bar{S}_i$ form a directed system of finite discrete spaces,
and we have the following lemmas.

\begin{lem}\label{inverse limit}
  $E$ is canonically homeomorphic to the inverse limit of this system.
\end{lem}

\begin{proof}
  Indeed, by abstract nonsense, from the existence and compatibility of the maps $E \to \bar{S}_i$,
  we know that there exists a continuous map
  $\Phi\colon E \to \varprojlim \bar{S}_i$.
  Since $E$ and the inverse limit are compact and Hausdorff,
  it suffices to show that $\Phi$ is a bijection.
  
  Recall that if $\varphi_i \colon \bar{S}_{i+1} \to \bar{S}_i$ is the map above,
  a point of the inverse limit is a tuple $(x_i)$ in the product $\prod_i \bar{S}_i$
  with the property that for each $i$, we have $\varphi_i(x_{i+1}) = x_i$.
  If $(x_i)$ is such a sequence,
  we may think of it in $E$
  as a choice, for each $i$, of a puncture or boundary component of $S_i$
  with the compatibility property that the puncture or boundary component chosen at the $(i+1)$st stage
  is contained in the boundary component chosen at the $i$th stage (or equal to if the $i$th choice is a puncture).
  This yields a chain of nonempty clopen subsets of $E$, say
  $U_1 \supset U_2 \supset \cdots$.
  Any point in the total intersection of these sets
  will be mapped by $\Phi$ to the sequence $(x_i)$,
  so the map $\Phi$ is surjective.
  By the (Exhaustion) property,
  if $x \ne y$ in $E$,
  every curve $\gamma$ which separates $x$ from $y$ is interior to some $S_i$,
  and thus the $\Phi$-images of $x$ and $y$ are distinct.

\end{proof}

\begin{lem}\label{spheres to spheres}
  Suppose that $E$ is infinite and second countable.
  If $\Phi\colon \mathscr{C}(E) \to \mathscr{C}(E')$ is an isomorphism of complexes of cuts
  and $S_1,S_2,\ldots$ is a principal spherical exhaustion of $E$,
  the cuts comprising each sphere $S_i$
  assemble into a sphere in $E'$ with the same number of punctures and boundary components.
  Moreover, these spheres are a principal spherical exhaustion
  $\Phi(S_1),\Phi(S_2),\ldots$ of $E'$,
  which is also infinite and second countable.
\end{lem}

\begin{proof}
  The graphs $\mathscr{C}(E)$ and $\mathscr{C}(E')$ have the same cardinality,
  which is countably infinite just when $E$ (or $E'$) is infinite and second countable.
  
  We claim that a collection $\gamma_1,\ldots,\gamma_k$ of pairwise-compatible cuts in $E$
  is actually a sphere $S$ in $E$
  with $n$ punctures and $k$ boundary components satisfying
  $n + k \ge 4$ and the (Infinite Complement) hypothesis
  if and only if the following statements hold.
  \begin{enumerate}
  \item The adjacency graph $A(\{\gamma_1,\ldots,\gamma_k\})$ is complete.
  \item ${(L(\gamma_1) \cap \cdots \cap L(\gamma_k))}^\perp$ has $k+1$ components,
    exactly one of which is finite.
  \item If $k > 1$, then for each cut $\gamma_i = U_i \sqcup V_i$, both $U_i$ and $V_i$ are infinite.
  \end{enumerate}

  Embed $E$ in $S^2$ and realize $\gamma_1,\ldots,\gamma_k$ by disjoint Jordan curves.
  It is straightforward to see that these cuts form a sphere with $n$ punctures and $k$ boundary components
  exactly when $S^2 - \{\gamma_1,\ldots,\gamma_k\}$ is homeomorphic to a disjoint union
  of $k$ discs and a sphere with $k$ boundary components
  and the intersection of $E$ with the sphere component is a finite discrete set of size $n$.
  In this situation the adjacency graph is complete.
  Supposing further that this sphere satisfies (Infinite Complement),
  we see by \Cref{spheres before spheres}
  that ${(L(\gamma_1) \cap \cdots \cap L(\gamma_k))}^\perp$ has $k+1$ components
  and that exactly one component is finite.
  When $k > 2$, (Infinite Complement) moreover \emph{implies} the third item above,
  while when $k = 2$, we have a sphere with $n$ punctures and $k$ boundary components
  when it is the disc complements in $S^2$ which have infinitely many points of $E$.
  Therefore if we have a sphere, then the conditions above hold.

  Suppose now that the conditions above hold.
  It is clear that these conditions are sufficient when $k = 1$, so suppose $k > 1$.
  Since both $U_i$ and $V_i$ are infinite for each $\gamma_i$,
  if we have a sphere system, it will satisfy (Infinite Complement).
  Drawing the cuts as Jordan curves on $S^2$,
  observe that the hypothesis on the adjacency graph implies
  that we have that $S^2 - \{\gamma_1,\ldots,\gamma_k\}$ is homeomorphic to a disjoint union
  of $k$ discs and a sphere with $k$ boundary components.
  Since we assume one component of ${(L(\gamma_1) \cap \cdots \cap L(\gamma_k))}^\perp$
  is finite, the assumption that each $U_i$ and $V_i$ are infinite
  implies that that component must correspond to the sphere with $k$ boundary components.
  As we saw in the proof of \Cref{opposite graph},
  in order for ${(L(\gamma_1) \cap \cdots \cap L(\gamma_k))}^\perp$ to have the full $k + 1$ components,
  the closure of the component of $S^2 - (E \cup \gamma_1 \cup \cdots \cup \gamma_k)$
  corresponding to this finite component of the opposite graph
  must have an essential simple closed curve,
  which requires a total number of punctures and boundary components satisfying $n + k \ge 4$.

  Now, each of the three conditions above is,
  one sees,
  preserved by $\Phi$,
  so the image of a collection of cuts in $\mathscr{C}(E)$ forming a sphere in $E$
  with $n$ punctures and $k$ boundary components
  is a sphere with $k$ boundary components.
  In fact, the quantity $n+k$ is preserved,
  since it can be read off of the dimension of the flag completion of the opposite graph
  of the finite component of ${(L(\gamma_1) \cap \cdots \cap L(\gamma_k))}^\perp$:
  we have $n + k - 4$ is equal to that dimension.

  Therefore if $S_1,S_2,\ldots$
  is a principal spherical exhaustion of $E$,
  we have spheres $\Phi(S_1),\Phi(S_2),\ldots$ in $E'$.
  The properties (Increasing) and (Exhaustion) and (Complexity)
  are clearly preserved by $\Phi$,
  and the lemma follows.
\end{proof}

\begin{proof}[Proof of \Cref{isomorphisms are geometric}]
  Observe that the graphs $\mathscr{C}(E)$ are finite
  if and only if $E$ is finite,
  so we may suppose that either both $E$ and $E'$
  are finite or both are infinite.

  Suppose at first that $E$ and $E'$ are finite.
  Since the dimension of $\mathscr{C}(E)$
  as a simplicial complex is $|E| - 4$
  when $E \ge 4$,
  we see that $|E| = |E'|$.
  We want to produce from $\Phi$ a bijection $f\colon E \to E'$
  whose action on cuts agrees with $\Phi$.

  In the case when $|E| = |E'| = 4$,
  the ``graphs'' $\mathscr{C}(E)$ and $\mathscr{C}(E')$
  are finite sets of three points,
  each of which corresponds, after choosing
  a basepoint $\star$ in $E$ and $\star'$ in $E'$
  to a two-element set containing $\star$ or $\star'$, respectively.
  There is therefore a bijection of $E$ with $E'$
  sending $\star$ to $\star'$
  and the point $a$ making up, for example the cut $\{\star, a\}$
  to the element $a'$ making up the cut corresponding to $\Phi(\{\star, a'\})$.
  Therefore the theorem holds for $E$ having size four;
  we therefore assume that $|E| \ge 5$.

  Let $\gamma$ and $\gamma'$ be outermost cuts
  which are not compatible,
  say $\gamma  = \{a, b\} \sqcup (E - \{a, b\})$
  and $\gamma' = \{a, c\} \sqcup (E - \{a, c\})$,
  and consider the peripheral pair
  $\delta = \{b, c\} \sqcup (E - \{b, c\})$
  and $\eta = \{a, b, c \} \sqcup (E - \{a, b, c\})$.
  In fact, notice that $\eta$ forms a peripheral pair
  with $\gamma$ and $\gamma'$ as well.

  Now, because $\Phi$ sends outermost cuts
  to outermost cuts and peripheral pairs to peripheral pairs
  and preserves compatibility,
  we conclude that $\Phi(\gamma) = \{a', b'\} \sqcup E' - \{ a', b' \}$,
  that $\Phi(\gamma') = \{a', c'\} \sqcup E' - \{a', c'\}$,
  that $\Phi(\eta) = \{ a', b', c' \} \sqcup E' - \{a', b', c'\}$
  and therefore that $\Phi(\delta) = \{b', c'\} \sqcup E - \{b', c'\}$.
  In other words,
  the \emph{pattern of intersection}
  of $\gamma$, $\gamma'$ and $\delta$
  is faithfully preserved by $\Phi$.
  Define, therefore, a map $f\colon E \to E'$
  by the rule that $f(a) = a'$,
  where $\gamma$ and $\gamma'$ are outermost cuts such that
  the intersection of the smaller sides of $\gamma$ and $\gamma'$ in $E$ is the singleton $\{a\}$
  and whose $\Phi$-images intersect in the singleton $\{a'\}$.

  Observe that this is well-defined independent of the choice of $\gamma$
  and $\gamma'$.
  For indeed,
  if $\gamma_1$ and $\gamma_1'$ are two other cuts
  whose smaller sides intersect in $\{a\}$,
  observe that under $\Phi$,
  the smaller sides of $\Phi(\gamma)$ and $\Phi(\gamma_1)$
  as well as $\Phi(\gamma')$ and $\Phi(\gamma_1')$ must intersect.
  In fact, the point of intersection
  must be $a'$:
  if $\Phi(\gamma_1) = \{b', a''\} \sqcup (E' - \{b', a'' \})$
  and $\Phi(\gamma_1') = \{c', a''\} \sqcup (E' - \{c', a'' \})$,
  notice that $\Phi(\gamma_1)$ is compatible with $\Phi(\gamma')$,
  in contradiction to what is true of $\gamma_1$ and $\gamma'$.

  The map $f$ must be injective and therefore a bijection:
  if $f(a) = f(b)$,
  we would have that the smaller sides of $\Phi(\gamma')$
  and $\Phi(\delta)$ as well as the smaller sides of $\Phi(\gamma)$
  and $\Phi(\delta)$ intersect in the same point,
  namely $f(a) = f(b)$.
  But from this, we must conclude that $\Phi(\gamma') = \Phi(\delta)$,
  since otherwise we would have that $f(c)$ is not well-defined.
  This contradicts the fact that $\Phi$ is an isomorphism.

  Indeed, we see from the construction of $f$
  and the proof that it is injective
  that in fact the push-forward action of $f$ on cuts
  agrees with $\Phi$.

  So much for the finite case.
  Suppose that $E$ is infinite and second countable.
  Then $E'$ is as well.
  Let $S_1, S_2,\ldots$ be a principal spherical exhaustion of $E$.
  By \Cref{spheres to spheres},
  the boundary cuts of each sphere
  assemble into a principal spherical exhaustion
  $\Phi(S_1),\Phi(S_2),\ldots$.

  Consider $S_{i+1}$ and choose a boundary component of $S_i$;
  it divides $S_{i+1}$ into two spheres.
  By (Complexity) if this sphere $S'$ has $n$ punctures and $k$ boundary components,
  we have $n + k \ge 5$.
  If $U_1,\ldots,U_k$ are the ``outer sides'' of the boundary components of this sphere,
  Consider the space $\bar{S'}$
  formed from $E$ by collapsing each of these $U_i$ to a point.
  This is a finite discrete space of size $n + k$.
  The same is true for $\overline{\Phi(S')}$.
  If $K^\perp$ is the finite component of ${(L(\gamma_1) \cap \cdots \cap L(\gamma_k))}^\perp$,
  the action of $\Phi$ restricts to an isomorphism between
  $K = {(K^\perp)}^\perp$ and $\Phi(K)$, the opposite graph
  of the finite component of ${(L(\Phi(\gamma_1)) \cap \cdots \cap L(\Phi(\gamma_k)))}^\perp$.

  The graph $K$ is isomorphic to $\mathscr{C}(\bar{S'})$,
  as $\Phi(K)$ is to $\mathscr{C}(\overline{\Phi(S')})$.
  By the finite case, the isomorphism
  $\Phi|_K\colon K \to \Phi(K)$ is induced by a bijection $\varphi'\colon \bar{S'} \to \overline{\Phi(S')}$.

  Indeed, a similar story holds true for each $S_i$ directly---namely,
  there is a bijection $\varphi_i\colon \bar{S}_i \to \overline{\Phi(S_i)}$
  inducing the isomorphism of finite components.
  We claim that the maps $\varphi_i$ are compatible with the restrictions $\bar{S}_{i+1} \to \bar{S}_i$
  in the sense that the following diagrams commute
  \[
    \begin{tikzcd}
      \bar{S}_{i+1} \ar[d, "\varphi_{i+1}"] \ar[r] & \bar{S}_i \ar[d, "\varphi_i"] \\
      \overline{\Phi(S_{i+1})} \ar[r] & \overline{\Phi(S_i)}.
    \end{tikzcd}
  \]
  Consider the collection of cuts
  formed by the boundary components of $S_{i+1}$ and $S_i$,
  and let $A$ denote the adjacency graph of these cuts.
  By \Cref{spheres before spheres},
  the subgraph comprising the boundary components of $S_i$ is complete,
  as is each of the subgraphs comprising one of the spheres $S'$ as above.
  Since by (Complexity) each boundary component of $S_i$ divides $S_{i+1}$ into two spheres,
  the corresponding vertex of $A$ is a cut vertex---removing it disconnects the graph into two components.
  In fact, one may argue directly that after replacing each of these vertices with two vertices, the graph
  $A$ is a disjoint union of $k+1$ complete graphs,
  where $k$ is the number of boundary components of $S_i$.
  Similarly one may argue that $L^\perp$,
  the opposite graph of the intersection of links,
  has $k+1$ finite components,
  one corresponding to $S_i$,
  and $k$ corresponding to the boundary components of $S_i$.
  Each component of the cut vertex decomposition of $A$
  corresponds uniquely to a finite component of $L^\perp$.

  Let $\pi$ denote the map $\bar{S}_{i+1} \to \bar{S}_i$
  and the map $\overline{\Phi(S_{i+1})} \to \overline{\Phi(S_i)}$.
  This map is ``the identity'' on punctures in the domain which come from punctures in the range.
  Otherwise, if $x$ is a puncture or boundary component in the domain which
  does not come from a puncture in the range,
  there is a unique boundary component in the range
  whose ``outer side'' contains $x$,
  and the map $\pi$ picks the point corresponding to that ``outer side''.
  
  The maps $\varphi_i$ and $\varphi_{i+1}$ send points corresponding to boundary components
  to points corresponding to boundary components
  and thus points corresponding to punctures to punctures.
  Moreover, $\Phi$ and thus $\varphi_{i+1}$ preserves the properties of $A$ and $L^\perp$ discussed above,
  so arguing as in the finite case,
  we see that the diagram above commutes.

  Since by \Cref{inverse limit},
  the spaces $E$ and $E'$ are inverse limits of these principal spherical exhaustions
  and we have bijections $\varphi_i$ which are compatible with the restriction maps,
  these bijections assemble into a homeomorphism $\varphi\colon E \to E'$.
  By the (Exhaustion) property, the agreement of the action of $\varphi$ with $\Phi$
  on each cut may be checked in some finite sphere, so follows from the finite case.
\end{proof}

\section{Faithfulness of the Action}

To complete the proof of \Cref{maintheorem},
we need to show that $\homeo(E)$
acts faithfully on $\mathscr{C}(E)$
when $E$ has at least five elements.
This, together with the foregoing theorem,
shows that the map $\homeo(E) \to \aut(\mathscr{C})$
is a continuous bijective homomorphism.
To show that it is an isomorphism of topological groups,
we need to know that the map is open.

One way to show this latter property is to say that if $K \subset U$
is compact and $U$ is open,
there exists a finite collection
$G = \{\gamma_1,\ldots,\gamma_k\}$
of vertices of $\mathscr{C}(E)$ such that $V(K,U)$
contains $\Stab(G)$;
that is, if $f \in \homeo(E)$
maps $K$ into $U$,
then actually the action of $f$ on cuts fixes each $\gamma_i$ as well.

Seeing that $\homeo(E)$ acts faithfully is not hard:
supposing that $f \in \homeo(E)$ is not the identity
but acts by the identity on $\mathscr{C}(E)$,
there exists some $x \in E$ such that $f(x) \ne x$.
Because $E$ has at least five elements,
there exist compatible non-peripheral cuts $\gamma = U \sqcup V$
and $\eta = U' \sqcup V'$ with the following properties:
\begin{enumerate}
\item $x \in U$ and $f(x) \in U'$.
\item $U \cap U' = \varnothing$,
\item $V \cap V' \ne \varnothing$.
\end{enumerate}
Because $f(x) \ne x$,
we have that if $f(\gamma) = \gamma$ and $f(\eta) = \eta$,
then in fact $f(U) = V$ and similarly $f(V') = U'$,
hence $f(U') = V'$ and $f(V) = U$.
But this is impossible, since $U \cap U'$ is empty while $V \cap V'$ is not.

\begin{proof}[Proof of \Cref{maintheorem}]
  By the remarks above, we have that $\homeo(E) \to \aut(\mathscr{C})$ is continuous and bijective.
  By \Cref{open lemma} below, it is also open, hence a homeomorphism.
  (When $E$ is finite, we have a continuous bijection between \emph{discrete} (finite) groups,
  which is clearly a homeomorphism.)
\end{proof}

\begin{lem}\label{open lemma}
  Suppose that $E$ is second countable and infinite.
  If $K$ is compact and $U$ is clopen such that $K \subset U$,
  there exists a finite subset $G \subset \mathscr{C}$
  such that $V(K,U)$ in $\homeo(E)$
  contains $\Stab(G)$.
\end{lem}

\begin{proof}
  Since $K$ is compact,
  we may cover it with finitely many disjoint clopen sets $U_1,\ldots,U_n$.
  We may assume that these sets lie within $U$.
  Each of the sets $U_1, \ldots, U_n$ and $\bigcup_i U_i$
  (which we may assume without loss by shrinking $U$ is actually $U$)
  is half of a cut in $E$,
  but not all of these cuts are non-peripheral.
  Considering only those cuts which are non-peripheral,
  we see we have a sphere $S$ in $E$.
  When the complexity of this sphere is at least $5$,
  we claim that if an element $f$ of $\homeo(E)$
  fixes each boundary cut of $S$ and each interior cut of $S$,
  then $f \in V(K,U)$.
  (Indeed, some smaller finite number of cuts should suffice,
  but we have no need for optimality.)

  To see this, note that if $f$ fixes each boundary cut of $S$,
  it clearly descends to an automorphism of the interior cuts in $S$,
  hence a permutation of the set of punctures and boundary components of $S$.
  Since the total number of these is at least five,
  we have by the faithfulness of the action
  that the permutation induced by $f$ is trivial.
  It follows that $f \in V(K,U)$,
  for $f$ must send each $U_i$, which corresponds either to a puncture or a boundary component of $S$,
  to itself, hence clearly inside $U$.

  Supposing that $S$ has smaller complexity,
  we need to add additional cuts in order to guarantee that $f \in V(K,U)$.
  If $K$ is infinite, we may simply choose a finer collection of disjoint clopen sets $U_i$
  in order to make the complexity of the sphere grow.
  If $K$ (and without loss of generality $U$) is finite,
  add for each point of $K$ a peripheral pair.
  Any element of $\homeo(E)$ fixing the peripheral pair,
  we have seen,
  fixes the point.
  Repeating for each point of $K$,
  we conclude.
\end{proof}

\section*{Acknowledgements}
The authors thank the anonymous referee for an extremely thoughtful read and many helpful suggestions. The second author also thanks Santana Afton, George Domat and Hannah Hoganson for helpful conversations and their interest in this line of inquiry.

\bibliographystyle{alpha}
\bibliography{bib.bib}

\begin{thebibliography}{HMV19}

\bibitem[AFP17]{AFP2017}
Javier Aramayona, Ariadna Fossas, and Hugo Parlier.
\newblock Arc and curve graphs for infinite-type surfaces.
\newblock {\em Proceedings of the American Mathematical Society}, 145(11):pp. 4995--5006, 2017.

\bibitem[Are46a]{Arens}
Richard Arens.
\newblock Topologies for homeomorphism groups.
\newblock {\em American Journal of Mathematics}, 68(4):593--610, 1946.

\bibitem[Are46b]{ArensAnnals}
Richard~F. Arens.
\newblock A topology for spaces of transformations.
\newblock {\em Ann. of Math. (2)}, 47:480--495, 1946.

\bibitem[Cam96]{Cameron}
Peter~J. Cameron.
\newblock Metric and topological aspects of the symmetric group of countable degree.
\newblock {\em European Journal of Combinatorics}, 17(2):135--142, 1996.

\bibitem[Dij05]{Dijkstra}
Jan~J. Dijkstra.
\newblock On homeomorphism groups and the compact-open topology.
\newblock {\em The American Mathematical Monthly}, 112(10):910--912, 2005.

\bibitem[HMV18]{Valdez1}
Jes\'{u}s Hern\'{a}ndez\space{}Hern\'{a}ndez, Israel Morales, and Ferr\'{a}n Valdez.
\newblock Isomorphisms between curve graphs of infinite-type surfaces are geometric.
\newblock {\em Rocky Mountain J. Math.}, 48(6):1887--1904, 2018.

\bibitem[HMV19]{Valdez2}
Jes\'us Hern\'andez\space{}Hern\'andez, Israel Morales, and Ferr\'an Valdez.
\newblock The {A}lexander method for infinite-type surfaces.
\newblock {\em Michigan Math. J.}, 68(4):743--753, 2019.

\bibitem[Ric63]{Richards1963OnTC}
I.~Richards.
\newblock On the classification of noncompact surfaces.
\newblock {\em Transactions of the American Mathematical Society}, 106:259--269, 1963.

\bibitem[Sto36]{Stone}
M.~H. Stone.
\newblock The theory of representations for {B}oolean algebras.
\newblock {\em Trans. Amer. Math. Soc.}, 40(1):37--111, 1936.

\end{thebibliography}
\end{document}